\documentclass[11pt,a4paper]{article}
\usepackage{bbm}
\usepackage{CJK}
\usepackage{latexsym}
\usepackage{amssymb}
\usepackage{graphicx}
\usepackage{epsfig}
\usepackage{mathrsfs}
\usepackage{amssymb}
\usepackage{amsmath}
\usepackage[dvips]{color}
\usepackage{subfigure}
\usepackage{graphicx}
\usepackage{epsfig}
\usepackage{float}
\usepackage{caption}
\usepackage{times}
\usepackage{txfonts}
\usepackage{graphicx} %

\newcommand{\matr}[1]{\mathbf{#1}}

\newcommand{\bS}{\mathbb{S}^d}
\newcommand{\bs}{\mathbb{S}^2}
\newcommand{\bR}{\mathbb{R}}

\newcommand{\bP}{\mathbb{P}_{t}}

\newcommand{\bX}{\mathcal{X}_{N}}

\newcommand{\vx}{\mathbf{x}}

\newcommand{\sY}{\Yd}

\newcommand{\Yd}{Y^{}_{\ell, k}}

\newcommand{\dt}{d_{t}}

\newtheorem{theorem}{Theorem}
\newtheorem{lemma}{Lemma}
\newtheorem{corollary}{Corollary}

\newtheorem{definition}{Definition}
\newtheorem{remark}{Remark}
\newenvironment{proof}{{\bf Proof:}}{\hfill$\bull$\medskip}
\newcommand{\bull}{\vrule height 1.8ex width 1.0ex depth 0ex}

\begin{document}
\begin{center}
{\bf\LARGE A Note On Characterizations of Spherical
t-Designs}\footnote{The work is supported by NSF of China
(No.11301222,\,\, No.11226305), and  NSF of Guangdong Province(No.
S2012040007860)}.

\vspace{0.4in} {\small Congpei An\footnote{E-mail address:
tancpei@jnu.edu.cn,\,andbach@163.com}}

\vspace{0.1in}{\small Department of Mathematics, Jinan University,
Guangzhou 510632, China}

\end{center}

\noindent{\bf Abstract}  A set
$\mathcal{X}_{N}=\{\vx_1,\ldots,\vx_N\}$ of $N$ points on the unit
sphere $\bS,\,d\geq 2$ is a spherical $t$-design if the average of
any polynomial of degree at most $t$ over the sphere is equal to the
average value of the polynomial over $\mathcal{X}_{N}$. This paper
extends characterizations of spherical $t$-designs in \cite{ACSW}
from $\bs$ to general $\bS$. We show that for
$N\geq\dim(\mathbb{P}_{t+1})$,  $\bX$ is a stationary point set of a
certain non-negative quantity $A_{N,\,t}$ and a fundamental system
for polynomial space over $\bS$ with degree at most $t$, then $\bX$
is a spherical $t$-design. In contrast, we present that with $N \geq
\dim( \bP)$, a fundamental system $\bX$ is a spherical $t$-design if
and only if non-negative quantity $D_{N,\,t}$ vanishes. In addition,
the still unanswered questions about construction of spherical
$t$-designs are discussed.

\vspace{0.05in} \noindent \textbf{Keywords:} spherical $t$-design;
non-negative quantity; nonlinear equation; fundamental system
\vspace{0.05in} \noindent\\


\section{Introduction}
\setcounter{equation}{0} \setcounter{lemma}{0}
\setcounter{proposition}{0} \setcounter{corollary}{0}

Let $\mathcal{X}_{N}=\left\{ {\mathbf{x}}_{1},\ldots ,{\mathbf{x}}%
_{N}\right\} $ be a set of $N$ points on the sphere~
$\mathbb{S}^{d}=\left\{~
\mathbf{x}\text{ }|~~\Vert \mathbf{x}\Vert _{2}=1\right\} $ $\subset \mathbb{R}%
^{d+1}$, and let $\mathbb{P}_{t}(\mathbb{S}^d)$ be the linear space
of restrictions of polynomials of degree at most $t$ in $d+1$
variables to $\mathbb{S}^{d}$. The set $\mathcal{X}_{N}$ is a
\emph{spherical $t$-design} if the average of any polynomial of
degree at most $t$ over the sphere is equal to the average value of
the polynomial over $\mathcal{X}_{N}$. That is
\begin{equation}\label{pnorm}
\frac{1 %
}{N}\sum_{j=1}^{N}p\left(
\mathbf{x}_{j}\right)=\frac{1}{\omega_d}\int_{\mathbb{S}^{d}}p(\mathbf{x})d\omega_d
(\mathbf{x})\quad\quad \forall p\in \mathbb{P}_{t}(\mathbb{S}^d),
\end{equation}
where $%
d\omega_{d} (\mathbf{x})$ denotes the surface measure on
$\mathbb{S}^{d},$ and $\omega_d$ is the surface area of the unit
sphere $\mathbb{S}^{d}$.
From (\ref{pnorm}),  it can be seen that $\mathcal{X}_{N}$ is a spherical $t$%
-design if the equal weight cubature rule (with weight $\omega_d/N$) is exact for all spherical
polynomials $p\in \mathbb{P}_{t}(\mathbb{S}^d)$.

Spherical $t$-designs was introduced by Delsarte, Goethals, and
Seidel \cite{DXY} in 1977. And the lower bound on $N$ was given as
the following \cite{DXY}:
\begin{eqnarray}\label{eq:Ntds}
N&\geq&{\displaystyle\left\{
\begin{array}{ll}2\binom{d+s}{d}, & \textrm{if}\ t=2s+1,
\\\binom{d+s}{d}+\binom{d+s-1}{d}, & \textrm{if}\ \
t=2s.\end{array}\right.}
\end{eqnarray}

In the past three decades, many works have been done on spherical
$t$-designs, see literature
\cite{ACSW}\cite{ACSW2}\cite{Bj}\cite{EBannai}\cite{EBnnai_09II}\cite{EBnnai_12}\cite{Hardin}\cite{Kore}\cite{Sloan}\cite{HSE}\cite{womweb}.
Seymour and Zaslavsky \cite{Seymour} showed that a spherical
$t$-design exists for any $t$ if $N$ is sufficiently large. In 2009,
there is a comprehensive survey of research on spherical $t$-designs
in the last three decades provided by Bannai and Bannai
\cite{EBannai}. As pointed out in \cite{Sloan}, $\bX$ is a spherical
$t$-design if and only if a certain non-negative quantity
$A_{N,\,t}$ is zero. Moreover, Sloan and Womersley \cite{Sloan} used
a condition based on the mesh norm to help determine if a stationary
point of $A_{N,\,t}$ is a spherical $t$-design. In particular for
$\bs$, {Chen and Frommer and Lang proved} that spherical $t$-designs
with $N=(t+1)^2$ is enough for all $t$ up to $t=100$. The well
conditioned spherical $t$-designs \cite{ACSW} have been applied into
interpolation, numerical integration and regularized least squares
approximation on $\bs$ \cite{ACSW2}. In very recently, the existence
of a spherical $t$-design for all $N\geq c t^d$ for some unknown
$c>0$ has been proved by Bondarenko, Radchenko and Viazovska
\cite{Bondarenko13}. In addition, Bondarenko, Radchenko and
Viazovska \cite{Bondarenko13_1} proved the existence of well
separated spherical $t$-designs for all $N\geq c t^d$ for some
unknown $c>0$. To our knowledge, there is no constructive proof that
spherical $t$-designs on general unit sphere $\bS$ with
$N\sim\dim(\bP)$. How to find a spherical $t$-design in general $
\bS$ for large $t$ is merit to be studied.

Spherical $t$-designs can be used to realize numerical interpolation
(when $N$ is right)\cite{ACSW}, integration
\cite{ACSW}\cite{Tagami}, hyperinterpolation \cite{sloan95},
filtered hyperinterpolation \cite{SW12} and regularized least
squares apprroximation \cite{ACSW2}. Among most of these
approximation schemes, it can be seen that the number of points should be
no less than the dimension of polynomial space. Therefore, it is
natural that we consider $N \geq \dim(\bP)$ in this paper.

 Our aim is to explore characterizations of spherical
$t$-designs on general unit sphere $\bS$ with the requirement that
$\bX$ is a fundamental system. We extend the \cite[Theorem
3.6]{ACSW} in $\bs$ based on \cite{ACSW} to general sphere $\bS$,
$d\geq2$.
On the other hand, we generalize the nonlinear approach to find a
spherical $t$-design for $N$ no less than the dimension of $\bP$ in
\cite{ACSW}. Furthermore, we show that a point set is a spherical
$t$-design if and only if a nonnegative quantity is zero when $N
\geq \dim(\bP)$.

In the next section we give some necessary materials of spherical
polynomials and known results. In Section 3, we present some results
on the unit sphere $\bS$. We sketch proofs in Section 4. Finally, in
Section 5 we discussion the known properties on spherical
$t$-designs, and draw attention to some still unknown problems.

\section{Preliminaries}
In the remainder of this paper, we follow  notation and terminology
in \cite{ACSW} and \cite{Sloan}.

Let $\{Y_{\ell ,k}\ :\ k=1,\ldots ,M(d,\ell),~\ell =0,1,\ldots,
t\}$ be an orthonormal set of (real) spherical harmonics, with $%
Y_{\ell ,k}$ a spherical harmonic of degree $\ell $ (see
\cite{Atkinson1} and \cite{Muller}).

As is known to all, we have
\begin{equation}
\begin{aligned}
&\bP:=\mathbb{P}_t(\bS)=\rm{span} \{\it{\sY}:\, k=1,\ldots,M(d,\ell),\,l=0,\ldots,t\} ,\\
&M(d,l)=\frac{(2l+d-1)(l+d-2)!}{(d-1)!l!},\\
&d_t:=\rm{dim}(\it
\bP)=\sum_{l=0}^{t}M(d,\,\ell)=M(d+1,\,t)\sim(t+1)^d,
\end{aligned}
\end{equation}
where $a_t\sim b_t$ means that positive constants $c_1,\,c_2$ exist,
independently of $t$, such that $c_1 a_t\leq b_t\leq c_2 a_t$. The
\emph{addition theorem} \cite{Muller}  for spherical harmonics is
\begin{equation}\label{eq:addTh}
\underset{k=1}{\overset{M(d,\,\ell)}{\sum }}Y_{\ell ,k}(\mathbf{x})Y_{\ell ,k}(%
\mathbf{y})=\frac{M(d,\,\ell)}{\omega_d }P_{\ell
}^{^{({d+1})}}\left( \mathbf{x}\cdot \mathbf{y}\right)
~~~~\forall~\text{ }\mathbf{x},\mathbf{y}\in \bS,
\end{equation}%
where $\mathbf{x}\cdot \mathbf{y}$ is the inner product in
$\mathbb{R}^{d+1}$ and $P_{\ell }^{({d+1})}$ is the Legendre
polynomial in $ \bR^{d+1}$ \cite{Szeg} of degree $\ell$ normalized
so that $P_{\ell }^{({d+1})}\left( 1\right) =1$.

For $t\geq 1,~N\geq d_t$ let the matrices $%
\mathbf{Y}_{t}^{0}$ and $\mathbf{Y}_{t}$ be defined by
\begin{equation}
\mathbf{Y}_{t}^{0}:=[Y_{\ell ,k}(\mathbf{x}%
_{j})],\qquad k=1,\dots,M(d,\ell),~~\ell =1,\dots,t;\qquad
j=1,\dots,N,\label{Yt0}
\end{equation}%
and
\begin{equation}
\mathbf{Y}_{t}:=\left[
\begin{array}{c}
\frac{1}{\sqrt{\omega_{d} }}\mathbf{e}^T \\
\mathbf{Y}_{t}^{0}%
\end{array}%
\right] \in \mathbb{R}^{{\dt}\times N},
\end{equation}%
where $\mathbf{e}=[1,\dots,1]^{T}\in \mathbb{R}^{N}.$

It is well know that $\bX$ a spherical $t$-design if and only if
\emph{Weyl sums} vanishes (see for example \cite{Sloan} and
\cite{DXY}), i.e.,
\begin{equation} \label{Weyl}
\underset{j=1}{\overset{N}{\sum }}Y_{\ell ,k}(\mathbf{x}%
_{j})=0,~~~~~~k=1,\dots,M(d,\,\ell),~~\ell=1,\dots, t.
\end{equation}

With the aid of \eqref{Yt0}, \eqref{Weyl} can be written in
matrix-vector form as
\begin{equation}\label{def:r(X)}
\mathbf{r}(\mathcal{X}_{N}):=\mathbf{Y}_{t}^{0}\mathbf{e}=%
\mathbf{0}.
\end{equation}%
where $\mathbf{r}(\mathcal{X}_{N})\in\mathbb{R}^{d_{t}-1}.$

Consequently, we can define the nonnegative quantity $A_{N,t}$ %
\begin{equation*}
A_{N,\,t}\left( \mathcal{X}_{N}\right) :=\frac{\omega_d }{N^{2}}\mathbf{r}(\mathcal{X}_{N})^{T}\mathbf{r}(%
\mathcal{X}_{N}).
\end{equation*}

The distance between any two points $\mathbf{x}$ and $\mathbf{y}$ on
the unit sphere $\bS$ is measured by the \emph{geodesic distance}
$dist(\mathbf{x}, \mathbf{y}):= \cos^{-1}(\mathbf{x}\cdot
\mathbf{y})\in[0,\pi]$.
\begin{definition}
  The mesh norm $h_{\mathcal{X}_N}$ of a point set
  $\mathcal{X}_N\subset\bS$ is
  \begin{equation}
h_{\mathcal{X}_{N}}:=\max_{{\mathbf{y}}\in \bS}\min_{\mathbf{x}_i%
\in \mathcal{X}_{N}}dist(\mathbf{y}, \mathbf{x}_i), \label{mesh}
\end{equation}
\end{definition}

\begin{definition}
  \label{def:fundsys}
  The set $\mathcal{X}_{N}\subset\bS$ is a fundamental
  system for $\mathbb{P}_t$ if the zero polynomial is the only
  element of $\mathbb{P}_t$ that vanishes at each point in
  $\mathcal{X}_{N}$, that is
  \begin{equation}
  p\in\mathbb{P}_t,~~~~p(\mathbf{x}_i)=0,~~~~~~i=1,...,N\label{eq:fundsysdef}
   \end{equation}
implies~~$p(\mathbf{x})\equiv0$ ~~for~~all~~
  $\mathbf{x}\in\bS$.
\end{definition}
For the case on $\bs$ \cite{ACSW}, the definition is also valid when
$N$ is large than the dimension of the polynomial space. In this
paper, we claim that this statement is also true on $\bS$, see Lemma
\ref{Lemmasing}.

As shown in \cite{Chen} \cite{CF} and \cite{ACSW}, one can find a
spherical $t$-design on $\bs$ by solving a system of underdetermined
nonlinear equations.  In this paper, we define the nonlinear
function $\matr{C}_{t}:(\bS)^N\to\bR^{N-1}$ as follows:

\begin{equation}\label{eq:Ct}
\matr{C}_{t}(\mathcal{X}_{N}):=\matr{EG}_{t}({\mathcal{X}_{N}})\mathbf{e},
\end{equation}%
where%
\begin{equation}
\matr{E}:=\left[\textbf{e},\,-\matr{I}_{N-1}
\right]\in\mathbb{R}^{(N-1)\times N}~~~\text{and}~~~ \matr{G}_{t}
:=\mathbf{Y}_{t} ^{T}\mathbf{Y}_{t}\in\bR^{N \times N}. \label{Gdef}
\end{equation}


\section{Theorems}

In this section we present some results on $\bS$, which are
generalized from $\bs$ in \cite{ACSW}.
\begin{lemma}
\label{Lemmasing}$\mathcal{X}_{N}$ is a fundamental system for
$\mathbb{P}_{t}$~if and only if~~$\mathbf{Y}_{t}^{{}}$\textit{\ is
of full row rank $\dim(\mathbb{P}_t)$  }.
\end{lemma}
\begin{theorem}
\label{Th.meshfund}\itshape\label{sd}If the mesh norm of the point
set $\mathcal{X}_N$ satisfies
 $h_{\mathcal {X}_{_{N}}}<\frac{1}{t}$, then $\mathcal{X}%
_{N}$ is a fundamental system for $\mathbb{P}_{t}$.
\end{theorem}
\begin{remark}
  Lemma \ref{Lemmasing} and Theorem \ref{Th.meshfund} can be proved
by the same way from \cite{ACSW}. Theorem \ref{Th.meshfund} shows
the condition of mesh norm is stronger that the condition of
fundamental system. The explain for $\bs$ is given by \cite{ACSW}.
\end{remark}
\begin{theorem}
\upshape\label{Th:variation}Let $t\geq 1$ and $N\geq \dim(\mathbb{P}_{t+1}).$ Let $%
\mathcal{X}_{N}\subset \bS$ be a stationary point of $A_{N,\,t}$. {%
Then} $\mathcal{X}_{N}$ is a spherical $t$-design, {or} there exists
a
non-zero polynomial $p\in \mathbb{P}_{t+1}$, such that $p\left( \mathbf{x}%
_{j}\right) =0$ for $j=1,\ldots ,N.$
\end{theorem}

By the definition of fundamental system (see Definition
\ref{def:fundsys}), we immediately have the following result.
\begin{corollary}
\label{ns} \label{MM}Let $t\geq 1$ and $$N\geq \dim(\mathbb{P}_{t+1}).$$ %
\color{black}Assume $\mathcal{X}_{N}\subset \mathbb{S}^{d}$ is a
stationary
point of $A_{N,\,t}$, and $\mathcal{X}_{N}$ is a fundamental system for $%
\mathbb{P}_{t+1.}$\color{black}~Then $\mathcal{X}_{N}$ is a spherical $t$%
-design.
\end{corollary}

\begin{theorem}
\label{Th:C} \upshape Let $N\geq d_t$. Suppose that $\mathcal{X}_{N}=\{%
\mathbf{x}_{1},...,\mathbf{x}_{N}\}$ is a fundamental system for $\mathbb{P}%
_{t}$.  $C_{t}\left( \mathcal{X}_{N}\right) =\mathbf{0}$ if and only if $%
\mathcal{X}_{N}$ is a spherical $t$-design.
\end{theorem}

Define the nonnegative quantity
\begin{equation}\label{eq:Dt}
D_{N,\,t}(\bX):=\frac{\omega_d^2}{N^2}\matr{C}_t(\bX)^{T}\matr{C}_t(\bX).
\end{equation}
We have the following result immediately.
\begin{corollary}
  \label{cor:cc}
\upshape Let $N\geq d_t$. Suppose that $\mathcal{X}_{N}=\{%
\mathbf{x}_{1},...,\mathbf{x}_{N}\}$ is a fundamental system for $\mathbb{P}%
_{t}$. Then
\begin{equation}
\label{ineq:Dtbound} 0\leq D_{N,\,t}(\bX)\leq 4(N-1)M^2(d+1,t),
\end{equation}
and $\bX$ is a spherical $t$-design if and only if
\begin{equation*}\label{eq:Dt0}
D_{N,\,t}(\bX)={0}.
\end{equation*}

\end{corollary}
\begin{proof}
From the definition of $\matr{C}_{t}(\bX)$, see \eqref{eq:Ct}, and
with the aid of addition theorem \eqref{eq:addTh}, we have
\begin{align}\label{eq:papp}
|(\matr{C}_{t}(\bX))_i|&=|(\matr{G}_t\matr{e})_1-(\matr{G}_t\matr{e})_{i+1}|\nonumber\\
&=\Big|\frac{1}{\omega_d}\sum_{j=1}^{N}\sum_{\ell=0}^{t}M(d,\,\ell)(P_{\ell}^{(d+1)}(\vx_1\cdot\vx_j))-P_{\ell}^{(d+1)}(\vx_{i+1}\cdot\vx_j))\Big|\\
&\leq\frac{2}{\omega_d}\sum_{j=1}^{N}\sum_{\ell=0}^{t}M(d,\,\ell)\nonumber\\
&=\frac{2N}{\omega_d}M(d+1,\,t),\nonumber \quad\quad i=1,\ldots,N-1.
\end{align}
Then by using the definition of \eqref{eq:Dt}, we obtain
\eqref{ineq:Dtbound}.

\end{proof}

\begin{remark}
  \label{remark:Dt}
  For the special case in $d=2$, \eqref{eq:papp} was show in
  \cite[Section 4]{Chen}.
\end{remark}

\section{Proofs}
For the completeness of this paper, we give proofs for Theorem
\ref{Th:variation} and Theorem \ref{Th:C}.
\subsection{Proof of Theorem\ref{Th:variation}}
This theorem rests on the following Lemma taken from \cite{Sloan}.
\begin{lemma}
\upshape\cite{Sloan}\label{le4} \itshape Let $t\geq 1$, and suppose $%
\mathcal{X}_{N}$ is a stationary point of $A_{N,\,t}$. Then either $\mathcal{X%
}_{N}$ is a spherical $t$-design, or there exists a non-constant polynomial $%
p\in \mathbb{P}_{t}$ with a stationary point at each point $\mathbf{x}%
_{i}\in \mathcal{X}_{N},i=1,\ldots,N.$
\end{lemma}

In the following for completeness we give the proof of Theorem
\ref{Th:variation}.

\begin{proof}
Suppose $\mathcal{X}_{N}$ is not a spherical $t$-design. Then by Lemma \ref%
{le4}, there exists a non-constant polynomial\textit{\
}$\mathit{q}\in
\mathbb{P}_{t}$ with a stationary point at each $\mathbf{x}_{i}\in \mathcal{X%
}_{N},i=1,\ldots ,N$, i.e.
\begin{equation}
\nabla ^{\ast }q(\mathbf{x}_{j})=0,~~~~~~j=1,\ldots,N,  \label{eq
33}
\end{equation}%
Now define
\begin{equation*}
p_{i}=\mathbf{e}_{i}\cdot \nabla ^{\ast }q,~~~~~~i=1,\ldots,d+1,
\end{equation*}%
where $\mathbf{e}_{1},\ldots,\mathbf{e}_{d+1}$ are the unit vectors
in the direction of the (fixed) coordinate axes for
$\mathbb{R}^{d+1}$, and the dot indicates the inner product in
$\mathbb{R}^{d+1}$.

By the stationary property of $q$, each $p_{i}$ for~$i=1,\ldots,d+1$
satisfies
\begin{equation*}
p_{i}(\mathbf{x}_{j})=0,~~~~~~j=1,\ldots,N.
\end{equation*}
Since $q$ is not a constant polynomial, at least one component of
$\nabla
^{\ast }q$ does not vanish identically, hence at least one of $%
p_{1},\ldots,p_{d+1}$ is not identically zero.\newline Assume
\begin{equation}
p:=p_{i_{0}}
\end{equation}
is not identically zero. Then because $q$ is a linear combination of
spherical harmonics $Y_{\ell,k}$
with $\ell=1,\ldots,t$ (see \cite{GW}, Chapter 12), then $p=p_{i_{0}}=\mathbf{e}%
_{i_{0}}\cdot \nabla ^{\ast }q$ is a linear combination of spherical
harmonics of degree $\ell-1$ and $\ell+1$. Thus for
$q\in\mathbb{P}_{t}$, then $p\in\mathbb{P}_{t+1}$.

Finally (\ref{eq 33}) gives
\begin{equation}
p(\mathbf{x}_{j})=0\text{ },~~~~ \text{}~~ j=1,\ldots,N,
\end{equation}
completing the proof.
\end{proof}

\subsection{Proof of Theorem\ref{Th:C}}
\begin{proof}
From (\ref{Gdef}), we have
\begin{equation*}
\matr{G}_{t} =\left[ \frac{1}{\sqrt{\omega_{d} }}\mathbf{e}%
\text{ }\,\,(\mathbf{Y}_{t}^{0}) ^{T}\right] %
\left[
\begin{array}{c}
\frac{1}{\sqrt{\omega_{d} }}\mathbf{e}^{T} \\
\mathbf{Y}_{t}^{0}
\end{array}%
\right] =\frac{1}{\omega_{d}} \mathbf{ee}^{T}+(\mathbf{Y}_{t}^{0})
^{T}\mathbf{Y}_{t}^{0}.
\end{equation*}%
Hence, from (\ref{eq:Ct}) and (\ref{def:r(X)}) we obtain
\begin{equation}
\matr{C}  _{t}\left( \mathcal{X}_{N}\right) =\frac{1}{\omega_d
}\matr{E}(\mathbf{Y}_{t}^{0})^{T}\mathbf{Y}_{t}^{{0}}
\mathbf{e}=\frac{1}{\omega_d }\matr{E}(\mathbf{Y}_{t}^{0})^{T}\mathbf{r}(\bX)%
 .  \label{C2}
\end{equation}
Let $\mathcal{X}_{N}=\{\mathbf{x}_{1},\dots,\mathbf{x}_{N}\}$ be a
fundamental system for $\mathbb{P}_{t}$.

Assume $\matr{C}_{t}\left( \mathcal{X}_{N}\right) =\mathbf{0}$, so we have
\begin{equation*}
\matr{E}(\mathbf{Y}_{t}^{0}) ^{T}\mathbf{r}\left(
\mathcal{X}_{N}\right) =\mathbf{0}.
\end{equation*}%
Then, all elements of $(\mathbf{Y}_{t}^{0})^{T}\mathbf{r}(%
\mathcal{X}_{N})$ are equal, i.e. there is a scalar $\nu $ such that
\begin{equation*}
\mathbf{(\mathbf{Y}}_{t}^{0})^{T}\mathbf{r}\left( \mathcal{X}%
_{N}\right) =\nu \mathbf{e}.
\end{equation*}%
This implies
\begin{equation*}
\left[ \frac{1}{\sqrt{\omega_{d} }}\mathbf{e}\text{ \
}\,(\mathbf{Y}_{t}^{0}) ^{T}\right] \left[
\begin{array}{c}
-\sqrt{\omega_d }\nu  \\
\mathbf{r}\left( \mathcal{X}_{N}\right)
\end{array}%
\right] =\mathbf{Y}_{t} ^{T}\left[
\begin{array}{c}
-\sqrt{\omega_d }\nu  \\
\mathbf{r}\left( \mathcal{X}_{N}\right)
\end{array}%
\right]=\mathbf{0}.
\end{equation*}%

Since $\mathbf{Y}_{t}^{T}$ is of full~(column)~rank, the only
solution is
\begin{equation*}
\nu =0,\qquad\mathbf{r}\left( \mathcal{X}_{N}\right) =\mathbf{0},
\end{equation*}%
$\mathcal{X}_{N}$ is a spherical $t$%
-design by following the matrix-vector form of Weyl sums is zero,
see (\ref{C2}). Conversely, suppose $\mathcal{X}_{N}$ is a spherical
$t$-design. By using \eqref{def:r(X)}, $\mathbf{r}\left(
\mathcal{X}_{N}\right) =\mathbf{0}$.
From (\ref{C2}) we have%
\begin{equation*}
\matr{C}_{t}\left( \mathcal{X}_{N}\right) =\mathbf{0}.
\end{equation*}
\end{proof}
\section{Discussion}
The geometry of a configuration on the unit sphere is a very
important issue when one considers numerical interpolation
\cite{ACSW}, potential theory \cite{HL}, and numerical integration
\cite{SW2000}\cite{HSW}. It is known that a spherical $t$-design
with a fixed number of points can have arbitrarily small minimum
distance between points (see \cite{HL}). Thus, a spherical
$t$-design can be with bad geometry \cite{ACSW} \cite{Kore}.
\begin{definition}
\label{def:welspds}
A point set $\bX\subset\bS$ is  \emph{well separated}, if the \emph{%
separation distance}
\begin{equation}\label{sep}
\delta _{\mathcal{X}_{N}}:=\min_{\mathbf{x}_{i},\mathbf{x}_{j}\in \mathcal{X}%
_{N},i\neq j}\rm{dist}\it\left( \mathbf{x}_{i},\mathbf{x}_{j}\right)
\geq \frac{c_d}{N^d}.
\end{equation}
\end{definition}
 The well conditioned spherical
designs is a well separated spherical designs investigated by
\cite{ACSW} for $N=(t+1)^2$ on $\bs$. The existence of well
separated spherical designs has been proved in recently
\cite{Bondarenko13_1}.

 It is known that mesh norm is the
covering radius for covering the sphere with spherical
caps of the smallest possible equal radius centered at the points in $%
\mathcal{X}_{N}$, while the separation distance $\delta
_{\mathcal{X}_{N}}$
is twice the packing radius, so $h_{\mathcal{X}_{N}}\geq \delta _{\mathcal{X}%
_{N}}/2$. As mentioned by \cite{HSW}, when the \emph{mesh ratio}
$\rho _{\mathcal{X}_{N}}$:
\begin{equation*}
\rho _{\mathcal{X}_{N}}:=\frac{2h_{\mathcal{X}_{N}}}{\delta _{\mathcal{X}_{N}}%
}\geq 1
\end{equation*}%
is smaller, the more uniformly are the points distributed on
$\mathbb{S}^{d}$. That is to say, mesh ratio can be regarded as a
good measure for the quality of the geometric distribution of
$\mathcal{X}_{N}$. For more information about point sets on sphere
and their applications, we refer to \cite{Saff_Kui} and \cite{GW}.

For choosing the point set $\bX$, if the points may be freely
chosen, then we shall see that there is merit in employing spherical
$t$-design to be nodes with some appropriate value of $t$ in
practical problems. Spherical $t$-designs have many applications:
interpolation, hyperinterpolation, numerical integrations, filtered
hyperinterpolations and regularized least squares approximations and
so on. Among these approximation schemes, especially on constructive
approximations, spherical $t$-design plays an irreplaceable role.
Thus, the study on how to construct the point set is really
necessary.

As is shown in \cite{Bj}\cite{EBannai}, the power of analytical
constructions for spherical $t$-designs is limited. It is merit to
study how to obtain (approximated) spherical $t$-designs by
numerical methods. For example, with the aid of nonlinear
optimization techniques, there are strong numerical results that
there exist spherical t-designs with close to $(t + 1)^2/2$ points
\cite{Sloan}\cite{womhalf}. Moreover, ``symmetric spherical
$t$-designs" \cite{womhalf} enjoy nice geometrical distribution.
This paper provides two ways to determinant spherical $t$-designs on
$\bS$ in fundamental systems for given $t$ with $N \geq\dim(\bP)$
points. It is evident that high dimensional optimization problems
have to be taken into account. Consequently, how to find a spherical
$t$-design on $\bS$ for a given large $t$ by reliable numerical
methods? This is a challenge problem even on $\bs$. Interval method
\cite{Alefeld}
 provided a useful way to guarantee there is a
very small neighborhood which contained a true spherical $t$-design
and computed spherical $t$-design \cite{Chen}\cite{CF}. Can we
extend this tool to general sphere $\bS$? This may be the thing to
be aimed at.

Clearly, a lot of work is need before we can claim to really
understand spherical $t$-design and its properties and applications.

\section{Acknowledgements}
 The author would like to acknowledge Eiichi Bannai and Yaokun Wu for their
encouragement on this work. He is indebted to Ian. H. Sloan and
Xiaojun Chen for discussions on spherical $t$-designs and to Rob S.
Womersley, Shuogang Gao, Takayuki Okuda, Makoto Tagami, Wei-Hsuan Yu
for stimulating suggestions.


{\small \begin{thebibliography}{99}
\bibitem{Alefeld} \textsc{G.~Alefeld and J.~Herzberger}, \emph{Introduction
to Interval Computations}, Computer Science and Applied
Mathematrics, Academic Press, New York, 1983.
\bibitem{ACSW} \textsc{C. An, X, Chen, I. H. Sloan, and R. S. Womersley}, \emph{Well conditioned spherical designs for integration
and interpolation on the two-sphere},   {SIAM J. Numer. Anal. 48
(2010), pp. 2135-2157}.
\bibitem{ACSW2}\textsc{C. An, X.Chen, I. H. Sloan and R. S. Womersley}, Regularized least squares approximation on the sphere using spherical designs.  SIAM J. Numer. Anal. 50 (2012), pp. 1513-1534.
.
\bibitem{Atkinson1}\textsc{K. E. Atkinson and W. Han}, Spherical Harmonics and Approximations on the Unit Sphere: An Introduction, Springer-Verlag, Berlin, 2012.


\bibitem{Bj} \textsc{B. Bajnok}, \emph{Construction of spherical t-designs},
Geom. Dedicata, 43 (1992), pp. ~167--179.

\bibitem{B} \textsc{E. Bannai}, \emph{Spherical designs and group
representations}, Contemp. Math., 34 (1984), pp. ~95--108.

\bibitem{EBannai} \textsc{E. Bannai and E. Bannai}, \emph{A survey on
spherical designs and algebraic combinatorics on spheres}, European
J. Combin., 30 (2009), pp. ~1392--1425.

\bibitem{EBnnai_09II}\textsc{E. Bannai and E. Bannai}, \emph{Spherical designs and Euclidean
designs. Recent developments in algebra and related areas},
pp.~1--37, Adv. Lect. Math. (ALM), 8, Int. Press, Somerville, MA,
2009.


\bibitem{EBnnai_12}\textsc{E. Bannai and E. Bannai}, \emph{ Remarks on the concepts of t-designs}.
J. Appl. Math. Comput. 40 (2012), no. 1-2, pp.~195--207.

\bibitem{R. Bauer} \textsc{R. Bauer}, \emph{Distribution of points on a
sphere with application to star catalogs} , J. Guidance, Control,
and Dynamics, 23 (2000), pp. ~130--137.

\bibitem{Bondarenko13}\textsc{A. Bondarenko, D. Radchenko and M. Viazovska}, \emph{Optimal asymptotic
bounds for spherical designs},  Ann. of Math. Issue 2, 178 (2013),
pp.~ 443--452.

\bibitem{Bondarenko13_1}\textsc{A. Bondarenko, D. Radchenko and M.
Viazovska}, \emph{Well-separated spherical designs},
arXiv:1303.5991.

\bibitem{CF} \textsc{X. Chen, A. Frommer and B. Lang}, \emph{Computational
existence proof for spherical t-designs}, Numer. Math., 117(2011),
pp.~ 289-305.

\bibitem{Chen} \textsc{X. Chen and R. S. Womersley}, \emph{Existence of
solutions to systems of undetermined equations and spherical
designs}, SIAM J. Numer. Anal., 44 (2006), pp. ~2326--2341.

\bibitem{DXY} \textsc{P. Delsarte, J. M. Goethals and J. J. Seidel}, \emph{%
Spherical codes and designs}, Geom. Dedicata, 6 (1977), pp.
~363--388

\bibitem{GW} \textsc{W. Freeden, T. Gervens and M. Schreiner}, \emph{%
Constructive Approximation on the Sphere and Application to
Geomathematics}, Clarendon Press, Oxford, 1998.


\bibitem{Hardin} \textsc{R. H. Hardin and N. J. A. Sloane}, \emph{McLaren's
improved snub cube and other new spherical designs in three
dimensions}, Discr. Comput. Geom., 15 (1996), pp. ~429--441.

\bibitem{HL} \textsc{K. Hesse and P. Leopardi}, \emph{The Coulomb energy of
spherical energy designs on $S^2$}, Adv. Comput. Math., 28 (2008),
pp. ~331--354.


\bibitem{HSW} \textsc{K. Hesse, I. H. Sloan and R. S. Womersley}, \emph{%
Numerical Integration on the Sphere,``Handbook of Geomathematics" },
edited by Willi Freeden, Zuhair Nashed and Thomas Sonar, Springer
Verlag, Berlin: 2010.

\bibitem{Kore} \textsc{J. Korevaar and J. L. H. Meyers}, \emph{Spherical
faraday cage for the case of equal point charges and Chebyshev-type
quadrature on the sphere}, J. Integral Transforms Special. Funct., 1
(1993), pp. ~105--117.

%
%

\bibitem{Muller} \textsc{C. M\"{u}ller}, \emph{%
Spherical Harmonics}, vol. 17 of Lecture Notes in Mathematics,
Springer Verlag, Berlin, New-York, 1966.

\bibitem{Reimer} \textsc{M. Reimer}, \emph{Multivariate Polynomial
Approximation}, Birkh\"{a}user Verlag, Berlin, 2003.

\bibitem{Reimer1} \leavevmode\vrule height 2pt depth -1.6pt width 23pt,
\emph{Interpolation on the sphere and bounds for the Lagrangian
square sums,} Resultate Math., 11 (1987), pp. ~144--166.

%
\bibitem{Saff_Kui} \textsc{E. B. Saff and A. B. J. Kuijlaars}, \emph{Distributing many points on a
sphere}, Math. Intelligencer, 19 (1997), pp. 5¨C11.
\bibitem{Seymour} \textsc{P. D. Seymour and T. Zaslavsky}, \emph{Averaging
sets: A generalization of mean values and spherical designs}, Adv.
Math., 52 (1984), pp. ~213--240.

\bibitem{sloan95}
\textsc{I. H. Sloan}, \emph{Polynomial interpolation and
hyperinterpolation over general regions}, J. Approx. Theory, 83
(1995), pp. 238--254.

\bibitem{SW2000} \textsc{I. H. Sloan and R. S. Womersley}, \emph{%
Constructive polynomial approximation on the sphere}, J. Approx.
Theory, 103 (2000), pp.~91--118.

\bibitem{SW12}\textsc{I. H. Sloan and R. S. Womersley}, \emph{Filtered hyperinterpolation: a constructive polynomial apporximation on the sphere}, Int.J. Geomath, 3 (2012), pp. 1-23.

\bibitem{G} \leavevmode\vrule height 2pt depth -1.6pt width 23pt, \emph{%
Extremal systems of points and numerical integration on the sphere},
Adv. Comput. Math., 21 (2004), pp. ~107--125.

\bibitem{Sloan} \leavevmode\vrule height 2pt depth -1.6pt width 23pt, \emph{%
A variational characterisation of spherical designs}, J. Approx.
Theory, 159 (2009), pp. ~308--318.

\bibitem{HSE} \textsc{N.~J.~A. Sloane}, \emph{Spherical Designs.} \newblock
\verb|http://neilsloane.com/sphdesigns/index.html|.

\bibitem{Smale} \textsc{S. Smale}, \emph{Mathematical problems for the next
century}, Math. Intelli., 20 (1998), pp. ~7--15.

\bibitem{Szeg} \textsc{G. Szeg\"{o}}, \emph{Orthogonal Polynomials},
In: American Mathematical Society Colloquium Publications, 4th ed.,
Volume 23, American Mathematical Society, Providence, Rhode Island,
1975.

\bibitem{Tagami} \textsc{M. Tagami}, \emph{Some applications of spherical
designs}, Mini-Workshop on Spherical Designs and Related Topics
November 19--21, 2012, Shanghai Jiao Tong University,
http://math.sjtu.edu.cn/conference/Bannai/2012/Slides/20121121Tagami.pdf

\bibitem{womweb} \textsc{R.~S. Womersley}, \emph{Interpolation and cubature
on the
sphere}\newblock\verb|web.maths.unsw.edu.au/~rsw/Sphere/Extremal/New/index.html|.
\bibitem{womhalf}\textsc{R. S. Womerlsey}, \emph{Spherical Designs with Close to the Minimal
Number of Points}, Applied Mathematics Report AMR09/26, Univeristy
of New South Wales, Sydney, Austrialia, 2009.
\bibitem{How} R. S. Womersley and I. H. Sloan, \emph{How good can polynomial
interpolation on the sphere be?}, Adv. Comput. Math., 14 (2001), pp.
~195--226.

\end{thebibliography}}
\end{document}